\documentclass{amsart}

\usepackage{amsmath,amssymb,amsfonts,amsbsy}
\usepackage{tikz-cd}
\usepackage{amsbsy}
\usepackage{amscd}
\usepackage{wasysym}
\usetikzlibrary{matrix,arrows,decorations.pathmorphing}
\usepackage{graphicx}
\usepackage{float}

\newcommand{\mc}[1]{\mathcal{#1}}
\newcommand{\e}[1]{\emph{#1}}

\newcommand{\la}{\langle}
\newcommand{\ra}{\rangle}

\newcommand{\rmv}[1]{}

\newcommand{\hs}{\hskip10pt}

\newcommand{\LG}{VN(G)}

\newcommand{\LO}{L^1(G)}
\newcommand{\LOQ}{L^1(\mathbb{G})}
\newcommand{\LOH}{L^{1}(\mathbb{H})}
\newcommand{\LOHH}{L^{1}(\widehat{\mathbb{H}})}
\newcommand{\LOQH}{L^1(\widehat{\mathbb{G}})}

\newcommand{\LTQ}{L^2(\mathbb{G})}
\newcommand{\LTH}{L^2(\mathbb{H})}

\newcommand{\LI}{L^{\infty}(G)}
\newcommand{\LIQ}{L^{\infty}(\mathbb{G})}
\newcommand{\LIH}{L^{\infty}(\mathbb{H})}
\newcommand{\LIQH}{L^{\infty}(\widehat{\mathbb{G}})}
\newcommand{\LIHH}{L^{\infty}(\widehat{\mathbb{H}})}

\newcommand{\LIQHP}{L^{\infty}(\widehat{\mathbb{G}}')}

\newcommand{\BLTQ}{\mc{B}(L^2(\mathbb{G}))}
\newcommand{\BLTH}{\mc{B}(L^2(\mathbb{H}))}

\newcommand{\vphi}{\varphi}

\newcommand{\lm}{\lambda}

\newcommand{\Gam}{\Gamma}
\newcommand{\om}{\omega}

\newcommand{\ten}{\otimes}
\newcommand{\oten}{\overline{\otimes}}
\newcommand{\hten}{\widehat{\otimes}}

\newcommand{\id}{\textnormal{id}}
\newcommand{\h}[1]{\widehat{#1}}
\newcommand{\wh}[1]{\widehat{#1}}

\providecommand{\norm}[1]{\lVert#1\rVert}

\newcommand{\G}{\mathbb{G}}
\newcommand{\Hb}{\mathbb{H}}
\newcommand{\C}{\mathbb{C}}

\newtheorem{thrm}{Theorem}[section]

\newtheorem{cor}[thrm]{Corollary}

\theoremstyle{definition}

\theoremstyle{remark}
\newtheorem{remark}[thrm]{Remark}

\numberwithin{equation}{section}

\begin{document}

\title{On hereditary properties of quantum group amenability}

\author{Jason Crann}
\address{School of Mathematics and Statistics, Carleton University, Ottawa, Ontario, Canada K1S 5B6}
\curraddr{Department of Pure Mathematics, University of Waterloo, Waterloo, Ontario, Canada N2L 3G1}
\curraddr{Institute for Quantum Computing, University of Waterloo, Waterloo, ON, Canada N2L 3G1}
\curraddr{Department of Mathematics and Statistics, University of Guelph, Guelph, ON, Canada N1G 2W1}
\email{jcrann@uwaterloo.ca}




\begin{abstract}
Given a locally compact quantum group $\G$ and a closed quantum subgroup $\Hb$, we show that $\G$ is amenable if and only if $\Hb$ is amenable and $\G$ acts amenably on the quantum homogenous space $\G/\Hb$. We also study the existence of $\LOQH$-module projections from $\LIQH$ onto $\LIHH$.
\end{abstract}

\maketitle

\section{Introduction}

Let $G$ be a locally compact group. If $G$ is amenable, it is well-known that any closed subgroup $H$ of $G$ is amenable and that $G$ acts amenably on the homogenous space $G/H$ in the sense that there is a $G$-invariant state $m\in L^{\infty}(G/H)^*$. Conversely, if $H$ is an amenable closed subgroup of $G$ for which $G$ acts amenably on $G/H$, then $G$ is necessarily amenable. We show that this fundamental hereditary property of amenability persists at the level of locally compact quantum groups using a recent homological characterization amenability \cite[Theorem 5.2]{C}, which states that a locally compact quantum group $\G$ is amenable if and only if its dual $\LIQH$ is 1-injective as an operator module over $\LOQH$.


The existence of (completely) bounded $A(G)$-module projections $P:\LG\rightarrow VN(H)$ has been a recent topic of interest in harmonic analysis \cite{DD,Der1,FKLS,KL}. In particular, it was shown in \cite[Theoerem 12]{Der1} that if $H$ is amenable, then a bounded $A(G)$-module projection always exists. We show that this projection may be taken completely contractive, and, in addition, that such a projection exists for a large class of locally compact quantum groups.

\section{Preliminaries}

Let $\mc{A}$ be a complete contractive Banach algebra. We say that an operator space $X$ is a right \e{operator $\mc{A}$-module} if it is a right Banach $\mc{A}$-module such that the module map $m_X:X\hten\mc{A}\rightarrow X$ is completely contractive, where $\hten$ denotes the operator space projective tensor product. We say that $X$ is \e{faithful} if for every non-zero $x\in X$, there is $a\in\mc{A}$ such that $x\cdot a\neq 0$. We denote by $\mathbf{mod}-\mc{A}$ the category of right operator $\mc{A}$-modules with morphisms given by completely bounded module homomorphisms. Left operator $\mc{A}$-modules and operator $\mc{A}$-bimodules are defined similarly, and we denote the respective categories by $\mc{A}-\mathbf{mod}$ and $\mc{A}-\mathbf{mod}-\mc{A}$.

Let $\mc{A}$ be a completely contractive Banach algebra and $X\in\mathbf{mod}-\mc{A}$. The identification $\mc{A}^+=\mc{A}\oplus_1\C$ turns the unitization of $\mc{A}$ into a unital completely contractive Banach algebra, and it follows that $X$ becomes a right operator $\mc{A}^+$-module via the extended action
\begin{equation*}x\cdot(a+\lm e)=x\cdot a+\lm x, \ \ \ a\in\mc{A}^+, \ \lm\in\C, \ x\in X.\end{equation*}
There is a canonical completely contractive morphism $\Delta^+:X\rightarrow\mc{CB}(\mc{A}^+,X)$ given by
\begin{equation*}\Delta^+(x)(a)=x\cdot a, \ \ \ x\in X, \ a\in\mc{A}^+,\end{equation*}
where the right $\mc{A}$-module structure on $\mc{CB}(\mc{A}^+,X)$ is defined by
\begin{equation*}(\Psi\cdot a)(b)=\Psi(ab), \ \ \ a\in\mc{A}, \ \Psi\in\mc{CB}(\mc{A}^+,X), \ b\in\mc{A}^+.\end{equation*}
An analogous construction exists for objects in $\mc{A}-\mathbf{mod}$. For $C\geq 1$, we say that $X$ is \e{relatively $C$-injective} if there exists a morphism $\Phi^+:\mc{CB}(\mc{A}^+,X)\rightarrow X$ such that $\Phi^+\circ\Delta^+=\id_{X}$ and $\norm{\Phi^+}_{cb}\leq C$. When $X$ is faithful, then $X$ is relatively $C$-injective if and only if there exists a morphism $\Phi:\mc{CB}(\mc{A},X)\rightarrow X$ such that $\Phi\circ\Delta=\id_{X}$ and $\norm{\Phi}_{cb}\leq C$ by the operator analogue of \cite[Proposition 1.7]{DP}, where $\Delta(x)(a):=\Delta^+(x)(a)$ for $x\in X$ and $a\in\mc{A}$. We say that $X$ is \e{$C$-injective} if for every $Y,Z\in\mathbf{mod}-\mc{A}$, every completely isometric morphism $\Psi:Y\hookrightarrow Z$, and every morphism $\Phi:Y\rightarrow X$, there exists a morphism $\widetilde{\Phi}:Z\rightarrow X$ such that $\norm{\widetilde{\Phi}}_{cb}\leq C\norm{\Phi}_{cb}$ and $\widetilde{\Phi}\circ\Psi=\Phi$, that is, the following diagram commutes:

\begin{equation*}
\begin{tikzcd}
Z \arrow[rd, dotted, "\widetilde{\Phi}"]\\
Y \arrow[u, hook, "\Psi"] \arrow[r, "\Phi"] &X
\end{tikzcd}
\end{equation*}

A \e{locally compact quantum group} is a quadruple $\G=(\LIQ,\Gam,\vphi,\psi)$, where $\LIQ$ is a Hopf-von Neumann algebra with co-multiplication $\Gam:\LIQ\rightarrow\LIQ\oten\LIQ$, and $\vphi$ and $\psi$ are fixed left and right Haar weights on $\LIQ$, respectively \cite{KV2,V}. For every locally compact quantum group $\G$, there exists a \e{left fundamental unitary operator} $W$ on $L_2(\G,\vphi)\ten L_2(\G,\vphi)$ and a \e{right fundamental unitary operator} $V$ on $L_2(\G,\psi)\ten L_2(\G,\psi)$ implementing the co-multiplication $\Gam$ via
\begin{equation*}\Gam(x)=W^*(1\ten x)W=V(x\ten 1)V^*, \ \ \ x\in\LIQ.\end{equation*}
Both unitaries satisfy the \e{pentagonal relation}; that is,
\begin{equation}\label{penta}W_{12}W_{13}W_{23}=W_{23}W_{12}\hs\hs\mathrm{and}\hs\hs V_{12}V_{13}V_{23}=V_{23}V_{12}.\end{equation}
By \cite[Proposition 2.11]{KV2}, we may identify $L_2(\G,\vphi)$ and $L_2(\G,\psi)$, so we will simply use $\LTQ$ for this Hilbert space throughout the paper. We denote by $R$ the unitary antipode of $\G$.

Let $\LOQ$ denote the predual of $\LIQ$. Then the pre-adjoint of $\Gam$ induces an associative completely contractive multiplication on $\LOQ$, defined by
\begin{equation*}\star:\LOQ\hten\LOQ\ni f\ten g\mapsto f\star g=\Gam_*(f\ten g)\in\LOQ.\end{equation*}
The canonical $\LOQ$-bimodule structure on $\LIQ$ is given by
\begin{equation*}f\star x=(\id\ten f)\Gam(x), \ \ \ x\star f=(f\ten\id)\Gam(x), \ \ \ x\in\LIQ, \ f\in\LOQ.\end{equation*}
A \e{Left invariant mean on $\LIQ$} is a state $m\in \LIQ^*$ satisfying
\begin{equation}\label{leftinv}\la m,x\star f\ra=\la f,1\ra\la m,x\ra, \ \ \ x\in\LIQ, \ f\in\LOQ.\end{equation}
Right and two-sided invariant means are defined similarly. A locally compact quantum group $\G$ is said to be \e{amenable} if there exists a left invariant mean on $\LIQ$. It is known that $\G$ is amenable if and only if there exists a right (equivalently, two-sided) invariant mean (cf. \cite[Proposition 3]{DQV}). We say that $\G$ is \e{co-amenable} if $\LOQ$ has a bounded left (equivalently, right or two-sided) approximate identity (cf. \cite[Theorem 3.1]{BT}).

The \e{left regular representation} $\lm:\LOQ\rightarrow\BLTQ$ of $\G$ is defined by
\begin{equation*}\lm(f)=(f\ten\id)(W), \ \ \ f\in\LOQ,\end{equation*}
and is an injective, completely contractive homomorphism from $\LOQ$ into $\BLTQ$. Then $\LIQH:=\{\lm(f) : f\in\LOQ\}''$ is the von Neumann algebra associated with the dual quantum group $\h{\G}$. Analogously, we have the \e{right regular representation} $\rho:\LOQ\rightarrow\BLTQ$ defined by
\begin{equation*}\rho(f)=(\id\ten f)(V), \ \ \ f\in\LOQ,\end{equation*}
which is also an injective, completely contractive homomorphism from $\LOQ$ into $\BLTQ$. Then $\LIQHP:=\{\rho(f) : f\in\LOQ\}''$ is the von Neumann algebra associated to the quantum group $\h{\G}'$. It follows that $\LIQHP=\LIQH'$, and the left and right fundamental unitaries satisfy $W\in\LIQ\oten\LIQH$ and $V\in\LIQHP\oten\LIQ$ \cite[Proposition 2.15]{KV2}. Moreover, dual quantum groups always satisfy $\LIQ\cap\LIQH=\LIQ\cap\LIQHP=\C1$ \cite[Proposition 3.4]{VD}.

If $G$ is a locally compact group, we let $\G_a=(\LI,\Gam_a,\vphi_a,\psi_a)$ denote the \e{commutative} quantum group associated with the commutative von Neumann algebra $\LI$, where the co-multiplication is given by $\Gam_a(f)(s,t)=f(st)$, and $\vphi_a$ and $\psi_a$ are integration with respect to a left and right Haar measure, respectively. The dual $\h{\G}_a$ of $\G_a$ is the \e{co-commutative} quantum group $\G_s=(\LG,\Gam_s,\vphi_s,\psi_s)$, where $\LG$ is the left group von Neumann algebra with co-multiplication $\Gam_s(\lm(t))=\lm(t)\ten\lm(t)$, and $\vphi_s=\psi_s$ is Haagerup's Plancherel weight. Then $L^1(\G_a)$ is the usual group convolution algebra $\LO$, and $L^1(\G_s)$ is the Fourier algebra $A(G)$. A commutative quantum group is amenable if and only if its underlying locally compact group is amenable, while co-commutative quantum groups are always amenable.


For general locally compact quantum groups $\G$, the left and right fundamental unitaries yield canonical extensions of the co-multiplication given by
\begin{align*}&\Gam^l:\BLTQ\ni T\mapsto W^*(1\ten T)W\in\LIQ\oten\BLTQ;\\
&\Gam^r:\BLTQ\ni T\mapsto V(T\ten 1)V^*\in\BLTQ\oten\LIQ.\end{align*}
In turn, we obtain a canonical $\LOQ$-bimodule structure on $\BLTQ$ with respect to the actions
$$f\rhd T=(\id\ten f)\Gam^r(T), \ \ \ T\lhd f=(f\ten\id)\Gam^l(T) \ \ \ f\in\LOQ, \ T\in\BLTQ.$$

\section{Closed Quantum Subgroups and Amenability}

Let $\G$ and $\Hb$ be two locally compact quantum groups. Then $\Hb$ is said to be a \e{closed quantum subgroup of $\G$ in the sense of Vaes} if there exists a normal, unital, injective *-homomorphism $\gamma:\LIHH\rightarrow\LIQH$ satisfying \begin{equation}\label{e:CQS}(\gamma\ten\gamma)\circ\Gam_{\h{\Hb}}=\Gam_{\h{\G}}\circ\gamma.\end{equation}
This is not the original definition of Vaes \cite[Definition 2.5]{V2}, but was shown to be equivalent in \cite[Theorem 3.3]{DKSS}.
With this definition, we have an analog of the Herz restriction theorem \cite{Herz} for quantum groups, that is, $\gamma_*:\LOQH\rightarrow\LOHH$ is a complete quotient map \cite[Theorem 3.7]{DKSS}. Indeed, if $\G$ and $\Hb$ are commutative, with underlying locally compact groups $G$ and $H$, the map $\gamma$ is nothing but the canonical inclusion
$$VN(H)\ni\lm_H(s)\mapsto\lm_G(s)\in VN(G),$$
where $\lm_H$ and $\lm_G$ are the left regular representations of $H$ and $G$, respectively. Its pre-adjoint $\gamma_*$ is then the canonical quotient map $A(G)\rightarrow A(H)$ arising from the classical Herz restriction theorem.

\begin{remark} There is an a priori weaker notion of closed quantum subgroup of a locally compact quantum group due to Woronowicz \cite[Definition 3.2]{DKSS}. In what follows, we restrict ourselves to Vaes' definition, so that a \textit{closed quantum subgroup} of a locally compact quantum group will always refer to the definition (\ref{e:CQS}) given above.
\end{remark}

Given a closed quantum subgroup $\Hb$ of a locally compact quantum group $\G$, we let $L^\infty(\G/\Hb):=\gamma(\LIHH)'\cap\LIQ$ represent the quantum homogenous space $\G/\Hb$. By the left version of \cite[Proposition 3.5]{KS} it follows that $\Gam(L^\infty(\G/\Hb))\subseteq L^\infty(\G/\Hb)\oten\LIQ$. Hence, $L^\infty(\G/\Hb)$ is a left operator $\LOQ$-submodule of $\LIQ$ under the action
$$f\star x=(\id\ten f)\Gam(x), \ \ \ x\in L^\infty(\G/\Hb), \ f\in\LOQ.$$
We say that $\G$ \textit{acts amenably on $\G/\Hb$} if there exists a state $m\in L^\infty(\G/\Hb)^*$ satisfying
$$\la m, f\star x\ra=\la f,1\ra\la m,x\ra, \ \ \ x\in L^\infty(\G/\Hb), \ f\in\LOQ.$$

Let $W^{\Hb}\in\LIH\oten\LIHH$ be the left fundamental unitary of $\Hb$. By the left version of \cite[Theorem 3.3]{DKSS} (see also \cite[Lemma 3.11]{KS}) it follows that the unitary $(\id\ten\gamma)(W^{\Hb})\in\LIH\oten\LIQH$ defines a co-action of $\Hb$ on $\G$ by
$$\alpha:\LIQ\ni x\mapsto(\id\ten\gamma)(W^{\Hb})^*(1_{\LIH}\ten x)(\id\ten\gamma)(W^{\Hb})\in\LIH\oten\LIQ,$$
where $\alpha$ satisfies $(\id\ten\alpha)\circ\alpha=(\Gam_{\Hb}\ten\id)\circ\alpha$. This co-action defines a right operator $\LOH$-module structure on $\LIQ$ by
$$x \star_{\Hb} g=(g\ten\id)\alpha(x), \ \ \ x\in\LIQ, \ g\in\LOH.$$
Moreover, the quantum homogenous space $L^{\infty}(\G/\Hb)$ is precisely the fixed point algebra $\LIQ^\alpha=\{x\in\LIQ\mid\alpha(x)=1\ten x\}$ \cite[Lemma 3.11]{KS}. Let $\alpha^l:\BLTQ\rightarrow\LIH\oten\BLTQ$ denote the canonical extension of $\alpha$, given by
$$\alpha^l(T)=(\id\ten\gamma)(W^{\Hb})^*(1_{\LIH}\ten T)(\id\ten\gamma)(W^{\Hb}), \ \ \ T\in\BLTQ.$$
We then obtain an extended right $\LOH$-module action on $\BLTQ$.

\begin{thrm}\label{t:amenablesubgroup} Let $\G$ and $\Hb$ be two locally compact quantum groups such that $\Hb$ is a closed quantum subgroup of $\G$. Then $\G$ is amenable if and only if $\Hb$ is amenable and $\G$ acts amenably on $\G/\Hb$.\end{thrm}

\begin{proof} If $\Hb$ is amenable and $\G$ acts amenably on $\G/\Hb$, there are states $m_{\Hb}\in\LIH^*$ and $m_{\G}\in L^\infty(\G/\Hb)^*$ which are invariant under the canonical $\LOH$ and $\LOQ$ actions on $\LIH$ and $L^\infty(\G/\Hb)$, respectively. In particular, $m_{\Hb}$ satisfies
$$(m_{\Hb}\ten\id)\Gam_{\Hb}(y)=\la m_{\Hb}, y\ra1, \ \ \ y\in\LIH.$$
Define the completely positive map $P:\LIQ\rightarrow\LIQ$ by
$$P(x)=(m_{\Hb}\ten\id)\alpha(x), \ \ \ x\in\LIQ.$$
Then, on the one hand,
\begin{align*}\alpha(P(x))&=\alpha((m_{\Hb}\ten\id)\alpha(x))=(m_{\Hb}\ten\id\ten\id)(\id\ten\alpha)\alpha(x)\\
&=(m_{\Hb}\ten\id\ten\id)(\Gam_{\Hb}\ten\id)\alpha(x)=1\ten ((m_{\Hb}\ten\id)\alpha(x))\\
&=1\ten P(x),\end{align*}
so $P(x)\in L^\infty(\G/\Hb)$. On the other hand, since $(\id\ten\gamma)(W^{\Hb})\in\LIH\oten\gamma(\LIQH)$, it follows that $P(axb)=aP(x)b$ for all $a,b\in L^\infty(\G/\Hb)=\gamma(\LIHH)'\cap\LIQ$. Thus, $P:\LIQ\rightarrow\LIQ$ is a completely positive projection onto $L^\infty(\G/\Hb)$. Moreover, as $V\in\LIQHP\oten\LIQ$, we have $(\id\ten\gamma)(W^{\Hb})_{12}V_{23}=V_{23}(\id\ten\gamma)(W^{\Hb})_{12}$, so that
\begin{align*}P(f\star_{\G} x)&=(m_{\Hb}\ten\id)\alpha((\id\ten f)V(x\ten 1)V^*)\\
&=(m_{\Hb}\ten\id)(\id\ten\id\ten f)(\alpha\ten\id)(V(x\ten 1)V^*)\\
&=(m_{\Hb}\ten\id)(\id\ten\id\ten f)((\id\ten\gamma)(W^{\Hb})_{12}^*V_{23}(x\ten 1)_{23}V_{23}^*(\id\ten\gamma)(W^{\Hb})_{12})\\
&=(m_{\Hb}\ten\id)(\id\ten\id\ten f)(V_{23}(\id\ten\gamma)(W^{\Hb})_{12}^*(x\ten 1)_{23}(\id\ten\gamma)(W^{\Hb})_{12}V_{23}^*)\\
&=(m_{\Hb}\ten\id)(\id\ten\id\ten f)(V_{23}(\alpha(x)\ten 1)V_{23}^*)\\
&=(\id\ten f)(V(P(x)\ten 1)V^*)\\
&=f\star_{\G}P(x),\end{align*}
for all $f\in\LOQ$ and $x\in\LIQ$. Thus, $P$ is a left $\LOQ$-module map. Defining $m:=m_{\G}\circ P\in\LIQ^*$, we obtain a right invariant mean on $\LIQ$, whence $\G$ is amenable.

Conversely, suppose that $\G$ is amenable. Clearly the restriction of a right invariant mean to $L^{\infty}(\G/\Hb)$ will be $\LOQ$-invariant, so $\G$ acts amenably on $\G/\Hb$. We show that $\Hb$ is amenable.

By \cite[Theorem 5.2]{C} $\LIQH$ is $1$-injective in $\LOQH-\mathbf{mod}$. The space $\BLTH=\mc{B}(L^2(\wh{\Hb}))$ becomes a left operator $\LOQH$-module via:
\begin{equation}\label{e:module}\hat{f}\rhd_{\h{\G}} T=(\id\ten\gamma_*(\hat{f}))V^{\h{\Hb}}(T\ten 1)V^{\h{\Hb}^*}, \ \ \ \hat{f}\in\LOQH, \ T\in\BLTH,\end{equation}
where $V^{\h{\Hb}}\in\LIH'\oten\LIHH$ is the right fundamental unitary of $\h{\Hb}$. Clearly, $\LIHH$ is an $\LOQH$-submodule of $\BLTH$ and $\gamma:\LIHH\rightarrow\LIQH$ is a left $\LOQH$-module map. Thus, we may extend $\gamma$ to a completely contractive left $\LOQH$-module map $\widetilde{\gamma}:\BLTH\rightarrow\LIQH$. Then  $\widetilde{\gamma}$ is a unital complete contraction and therefore completely positive. Moreover, $\LIHH$ is contained in the multiplicative domain of $\widetilde{\gamma}$ (as it extends a *-homomorphism), so the bimodule property of completely positive maps over their multiplicative domains (cf. \cite{Choi}) ensures that
$$\widetilde{\gamma}(\hat{x}T\hat{y})=\gamma(\hat{x})\widetilde{\gamma}(T)\gamma(\hat{y}), \ \ \ \hat{x},\hat{y}\in\LIHH, \ T\in\BLTH.$$
Thus, $\widetilde{\gamma}$ is a completely bounded $\LIHH$-$\gamma(\LIHH)$ bimodule map between $\BLTH$ and $\BLTQ$. By \cite[Theorem 2.5]{EK} we can approximate $\widetilde{\gamma}$ in the point weak* topology by a net $\widetilde{\gamma}_i$ of normal completely bounded $\LIHH$-$\gamma(\LIHH)$ bimodule maps. Since $\LIH$ is standardly represented on $\LTH$, for each $g\in\LOH$ there exist $\xi,\eta\in\LTH$ such that $g=\om_{\xi,\eta}|_{\LIH}$. Taking an orthonormal basis $(e_j)$ of $\LTH$ it follows that
$$T\lhd g=(\om_{\xi,\eta}\ten\id)W^{\Hb^*}(1\ten T)W^{\Hb}=\sum_j(\om_{e_j,\eta}\ten\id)(W^{\Hb^*})T(\om_{\xi,e_j}\ten\id)(W^{\Hb}),$$
for $T\in\BLTH$, where the series converges in the weak* topology. Thus,
\begin{align*}\widetilde{\gamma}(T\lhd_{\Hb} g)&=\lim_i\widetilde{\gamma}_i(T\lhd_{\Hb}g)\\
&=\lim_i\widetilde{\gamma}_i\bigg(\sum_j(\om_{e_j,\eta}\ten\id)(W^{\Hb^*})T(\om_{\xi,e_j}\ten\id)(W^{\Hb})\bigg)\\
&=\lim_i\sum_j\widetilde{\gamma}_i((\om_{e_j,\eta}\ten\id)(W^{\Hb^*})T(\om_{\xi,e_j}\ten\id)(W^{\Hb}))\\
&=\lim_i\sum_j\gamma((\om_{e_j,\eta}\ten\id)(W^{\Hb^*}))\widetilde{\gamma}_i(T)\gamma((\om_{\xi,e_j}\ten\id)(W^{\Hb}))\\
&=\lim_i\sum_j(\om_{e_j,\eta}\ten\id)((\id\ten\gamma)(W^{\Hb^*}))\widetilde{\gamma}_i(T)(\om_{\xi,e_j}\ten\id)((\id\ten\gamma)(W^{\Hb}))\\
&=\lim_i(\om_{\xi,\eta}\ten\id)(\id\ten\gamma)(W^{\Hb})^*(1\ten\widetilde{\gamma}_i(T))(\id\ten\gamma)(W^{\Hb})\\
&=\lim_i(g\ten\id)\alpha^l(\widetilde{\gamma}_i(T))\\
&=(g\ten\id)\alpha^l(\widetilde{\gamma}(T))\\
&=\widetilde{\gamma}(T)\lhd_{\Hb} g.\end{align*}
Thus, $\widetilde{\gamma}$ is a right $\LOH$-module map.

Since $\widetilde{\gamma}$ is also a left $\LOQH$-module map, for $x\in\LIH$ and $\hat{f},\hat{g}\in\LOQH$ we have
\begin{align*}\la\Gam_{\h{\G}}(\widetilde{\gamma}(x)),\hat{f}\ten\hat{g}\ra&=\la\widetilde{\gamma}(x),\hat{f}\star_{\h{\G}}\hat{g}\ra =\la\hat{g}\rhd_{\h{\G}}\widetilde{\gamma}(x),\hat{f}\ra\\
&=\la\widetilde{\gamma}(\hat{g}\rhd_{\h{\G}}x),\hat{f}\ra=\la\hat{g},1\ra\la\widetilde{\gamma}(x),\hat{f}\ra\\
&=\la\widetilde{\gamma}(x)\ten 1,\hat{f}\ten\hat{g}\ra.
\end{align*}
The standard argument then shows that $\widetilde{\gamma}(x)\in\LIQH\cap\LIQ=\C1$. Hence, $\widetilde{\gamma}|_{\LIH}$ defines a left invariant mean on $\LIH$, and $\Hb$ is amenable.
\end{proof}

\begin{remark} Theorem \ref{t:amenablesubgroup} yields a new proof of the classical fact that amenability of a locally compact group passes to closed subgroups \textit{without} the use of Bruhat functions or Leptin's theorem \cite{Lep} on the equivalence of amenability of $G$ and the existence of a bounded approximate identity in the Fourier algebra $A(G)$.\end{remark}

For a locally compact quantum group $\G$, the set of unitary elements $u\in\LIQ$ satisfying $\Gam(u)=u\ten u$ forms a locally compact group under the relative weak* topology, called the \e{intrinsic group} of $\G$, and is denoted $\mathrm{Gr}(\G)$. The \e{character group} of $\G$ is defined as $\widetilde{\G}:=\mathrm{Gr}(\widehat{\G})$. It follows that $\widetilde{\G_a}\cong G$, its underlying locally compact group, and $\widetilde{\G_s}\cong\widehat{G}$, the group of continuous characters on $G$. For more details on the character group and properties of the assignment $\G\mapsto\widetilde{\G}$ we refer the reader to \cite{D2,KN2}. Since $\widetilde{\G}$ is always a closed quantum subgroup of $\G$ \cite[Theorem 5.5]{D2}, we obtain the following generalization of \cite[Theorem 5.14]{KN2} beyond discrete quantum groups. We remark that the same conclusion was obtained in \cite[Theorem 5.6]{D2} under the a priori stronger hypothesis that $\widehat{\G}$ is co-amenable.

\begin{cor} Let $\G$ be a locally compact quantum group. If $\G$ is amenable then $\widetilde{\G}$ is amenable.\end{cor}

We now study the existence of completely contractive $\LOQH$-module projections $P:\LIQH\rightarrow\LIHH$. Our main tool is the following refinement of \cite[Theorem 5.1]{C}.

\begin{thrm}\label{t:subgroup} Let $\G$ and $\Hb$ be locally compact quantum groups such that $\h{\G}$ is amenable and $\Hb$ is a closed quantum subgroup of $\G$. Then $\Hb$ is amenable if and only if $\LIHH$ is 1-injective in $\mathbf{mod}-\LOQH$.\end{thrm}

\begin{proof} The canonical faithful right $\LOQH$-module action on $\LIHH$ is given by
$$\hat{x}\star_{\h{\G}}\hat{f}=\hat{x}\star_{\h{\Hb}}\gamma_*(\hat{f}), \ \ \ \hat{x}\in\LIHH, \ \hat{f}\in\LOQH.$$
The embedding $\Delta:\LIHH\rightarrow\mc{CB}(\LOQH,\LIHH)$ is then
$$\Delta(\hat{x})(\hat{f})=\hat{x}\star_{\h{\G}}\hat{f}, \ \ \ \hat{x}\in\LIHH, \ \hat{f}\in\LOQH.$$
Under the canonical identification $\mc{CB}(\LOQH,\LIHH)=\LIQH\oten\LIHH$, one easily sees that $\Delta=(\gamma\ten\id)\circ\Gam_{\h{\G}}$ and that the pertinent $\LOQH$-module structure on $\LIQH\oten\LIHH$ is
$$X\star_{\h{\G}}\hat{f}=(\hat{f}\ten\id\ten\id)(\Gam_{\h{\G}}\ten\id)(X), \ \ \ X\in\LIQH\oten\LIHH, \ \hat{f}\in\LOQH.$$
Let $U:=(\gamma\ten\id)(W^{\h{\Hb}})\in\LIQH\oten\LIH$, where $W^{\h{\Hb}}$ is the left fundamental unitary of $\h{\Hb}$. The pentagonal relation (\ref{penta}) together with equation (\ref{e:CQS}) imply
\begin{equation}\label{e:U}U_{23}W_{12}^{\h{\G}^*}=U_{13}^*W_{12}^{\h{\G}^*}U_{23},\end{equation}
where $W^{\h{\G}}$ is the left fundamental unitary of $\h{\G}$. By \cite[Theorem 5.1]{C}, amenability of $\Hb$ entails the existence of a completely contractive right $\LOHH$-module projection $E:\BLTH\rightarrow\LIHH$ with respect to the action
$$T\lhd_{\h{\Hb}}\hat{f}=(\hat{f}\ten\id)W^{\h{\Hb}^*}(1\ten T)W^{\h{\Hb}}, \ \ \ \hat{f}\in\LOQH, \ T\in\BLTH.$$
Let $m\in\LIQH^*$ be a left invariant mean, and define $\Phi:\LIQH\oten\LIHH\rightarrow\LIHH$ by
$$\Phi(X)=E((m\ten\id)(UXU^*)), \ \ \ X\in\LIQH\oten\LIHH.$$
Then for $\hat{x}\in\LIHH$ we have
$$\Phi(\Delta(\hat{x}))=E((m\ten\id)((\gamma\ten\id)(W^{\h{\Hb}}\Gam_{\h{\Hb}}(\hat{x})W^{\h{\Hb}^*})))=E((m\ten\id)(1\ten\hat{x}))=\hat{x},$$
so that $\Phi$ is a completely contractive left inverse to $\Delta$. Now, fix $\hat{f}\in\LOQH$ and $X\in\LIQH\oten\LIHH$. Then
\begin{align*}\Phi(X\star_{\h{\G}}\hat{f})&=E((m\ten\id)(U((\hat{f}\ten\id\ten\id)(\Gam_{\h{\G}}\ten\id)(X))U^*))\\
&=E((m\ten\id)(U((\hat{f}\ten\id\ten\id)(W_{12}^{\h{\G}^*}X_{23}W_{12}^{\h{\G}}))U^*))\\
&=E((m\ten\id)(\hat{f}\ten\id\ten\id)(U_{23}W_{12}^{\h{\G}^*}X_{23}W_{12}^{\h{\G}}U_{23}^*))\\
&=E((m\ten\id)(\hat{f}\ten\id\ten\id)(U_{13}^*W_{12}^{\h{\G}^*}U_{23}X_{23}U_{23}^*W_{12}^{\h{\G}}U_{13})) \ \ \ \textnormal{(by (\ref{e:U}))}\\
&=E((\hat{f}\ten\id)(\id\ten m \ten\id)(U_{13}^*W_{12}^{\h{\G}^*}U_{23}X_{23}U_{23}^*W_{12}^{\h{\G}}U_{13}))\\
&=E((\hat{f}\ten\id)(U^*((\id\ten m \ten\id)(W_{12}^{\h{\G}^*}(1\ten UXU^*)W_{12}^{\h{\G}}))U)).\\
\end{align*}
Using the fact that $m$ is a left invariant mean on $\LIQH$, it follows as in \cite[Theorem 5.5]{CN} that
$$(\id\ten m \ten\id)(W_{12}^{\h{\G}^*}(1\ten UXU^*)W_{12}^{\h{\G}})=(\id\ten m \ten\id)(1\ten UXU^*).$$
Thus,
\begin{align*}\Phi(X\star_{\h{\G}}\hat{f})&=E((\hat{f}\ten\id)(U^*((\id\ten m \ten\id)(1\ten UXU^*))U))\\
&=E((\hat{f}\ten\id)((\gamma\ten\id)(W^{\h{\Hb}^*})(1\ten (m\ten\id)(UXU^*))(\gamma\ten\id)(W^{\h{\Hb}})))\\
&=E((\hat{f}\ten\id)(\gamma\ten\id)(W^{\h{\Hb}^*}(1\ten (m\ten\id)(UXU^*))W^{\h{\Hb}}))\\
&=E((\gamma_*(\hat{f})\ten\id)(W^{\h{\Hb}^*}(1\ten (m\ten\id)(UXU^*))W^{\h{\Hb}}))\\
&=E(((m\ten\id)(UXU^*))\lhd_{\h{\Hb}}\gamma_*(\hat{f}))\\
&=E((m\ten\id)(UXU^*))\star_{\h{\Hb}}\gamma_*(\hat{f})\\
&=\Phi(X)\star_{\h{\Hb}}\gamma_*(\hat{f})\\
&=\Phi(X)\star_{\h{\G}}\hat{f}.\\
\end{align*}
Hence, $\Phi$ is a right $\LOQH$-module map, and it follows that $\LIHH$ is relatively 1-injective in $\mathbf{mod}-\LOQH$. Since $\LIHH$ is a 1-injective operator space, \cite[Proposition 2.3]{C} entails the 1-injectivity of $\LIHH$ in $\mathbf{mod}-\LOQH$.

Conversely, if $\LIHH$ is 1-injective in $\mathbf{mod}-\LOQH$ there exists a completely contactive right $\LOQH$-module projection $E:\BLTH\rightarrow\LIHH$. Then $E$ is a right $\LOHH$-module projection, which by \cite[Proposition 4.7]{C} entails the amenability of $\Hb$.\end{proof}

\begin{cor} Let $\G$ and $\Hb$ be locally compact quantum groups such that $\h{\G}$ is amenable and $\Hb$ is a closed amenable quantum subgroup of $\G$ in the sense of Vaes. Then there exists a completely contractive right $\LOQH$-module projection $P:\LIQH\rightarrow\LIHH$.\end{cor}

\begin{proof} By Theorem \ref{t:subgroup}, $\LIHH$ is 1-injective as a right $\LOQH$-module. The identity $\id_{\LIHH}$ therefore extends to a completely contractive morphism $P:\LIQH\rightarrow\LIHH$.\end{proof}

Since every co-commutative quantum group is amenable, the following corollary is immediate.

\begin{cor} Let $G$ be a locally compact group and let $H$ be a closed amenable subgroup of $G$. Then there exists a completely contractive $A(G)$-module projection $P:\LG\rightarrow VN(H)$.\end{cor}

\begin{remark} The existence of a completely bounded $A(G)$-module projection $P:\LG\rightarrow VN(H)$ was shown in \cite[Theorem 1.3]{FKLS} under the stronger assumption that $G$ is amenable.\end{remark}

\bibliographystyle{amsplain}

\end{document}